\documentclass[12pt,letterpaper]{article}

\usepackage{amssymb,amsfonts,amscd,amsthm}
\usepackage[all,arc]{xy}
\usepackage{enumerate, bbm}
\usepackage{mathrsfs}
\usepackage{tikz}
\usepackage{fullpage}
\usepackage{mathtools}
\usepackage{hyperref}
\usepackage{color}
\usepackage{graphicx}
\usepackage{enumitem}
\graphicspath{ {TexImages/} }

\newtheorem{thm}{Theorem}[section]
\newtheorem{cor}[thm]{Corollary}

\newtheorem{lem}[thm]{Lemma}

\newtheorem{quest}[thm]{Question}

\newtheorem{conj}[thm]{Conjecture}
\theoremstyle{definition}
\newtheorem{defn}[thm]{\protect\definitionname}

\providecommand{\definitionname}{Definition}

\hypersetup{
	colorlinks,
	citecolor=blue,
	filecolor=blue,
	linkcolor=blue,
	urlcolor=blue,
	linktocpage
}

\setenumerate[1]{label=\thesection.\arabic*.}
\setenumerate[2]{label*=\arabic*.} 
\setlist[itemize]{itemsep=0ex, topsep=1ex}

\newcommand{\E}{\mathbb{E}}

\newcommand{\be}{\beta}
\newcommand{\gam}{\gamma}

\newcommand{\sig}{\sigma}
\newcommand{\ep}{\varepsilon}
\newcommand{\lam}{\lambda}

\newcommand{\Om}{\Omega}

\newcommand{\del}{\delta}

\renewcommand{\l}{\left}
\renewcommand{\r}{\right}

\newcommand{\half}{\frac{1}{2}}
\newcommand{\quart}{\frac{1}{4}}

\renewcommand{\c}[1]{\mathcal{#1}}

\newcommand{\mf}[1]{\mathfrak{#1}}

\newcommand{\ol}[1]{\overline{#1}}
\newcommand{\tr}[1]{\textrm{#1}}
\newcommand{\rec}[1]{\frac{1}{#1}}
\newcommand{\f}[2]{\frac{#1}{#2}}
\newcommand{\floor}[1]{\l\lfloor #1\r\rfloor}
\newcommand{\ceil}[1]{\l\lceil #1\r\rceil}

\newcommand{\rand}[1]{\boldsymbol{#1}}

\renewcommand{\S}{\mf{S}}

\renewcommand{\Pr}{\mathbb{P}}

\title{Card Guessing with Partial Feedback}
\author{Persi Diaconis\footnote{Dept. of Mathematics and Statistics, Stanford University. Research supported by NSF Grant DMS-1954042.} \and Ron Graham\footnote{Dept. of Computer Science and Engineering, UCSD,  {\tt graham@ucsd.edu}. } \and Xiaoyu He \footnote{Dept. of Mathematics, Stanford University, {\tt alkjash@stanford.edu}. Research supported by NSF Graduate Research Fellowship Grant No. DGE-1656518.}\and Sam Spiro\footnote{Dept.\ of Mathematics, UCSD, {\tt sspiro@ucsd.edu}. Research supported by NSF Graduate Research Fellowship Grant No. DGE-1650112.}}
\date{\today}
\begin{document}
	\maketitle
	
	\begin{abstract}
		Consider the following experiment: a deck with $m$ copies of $n$ different card types is randomly shuffled, and a guesser attempts to guess the cards sequentially as they are drawn.  Each time a guess is made, some amount of ``feedback'' is given.  For example, one could tell the guesser the true identity of the card they just guessed (the complete feedback model) or they could be told nothing at all (the no feedback model).  
		
		In this paper we explore a partial feedback model, where upon guessing a card, the guesser is only told whether or not their guess was correct. We show in this setting that, uniformly in $n$, at most $m+O(m^{3/4}\log m)$ cards can be guessed correctly in expectation.  This resolves a question of Diaconis and Graham from 1981, where even the $m=2$ case was open.
	\end{abstract}
	\section{Introduction}
	
	Let $\S_{m,n}$ be the set of words $\pi$ over the alphabet $[n]:=\{1,2,\ldots,n\}$ where each character in $[n]$ appears exactly $m$ times in $\pi$. We think of $\pi$ as some way to shuffle a deck of cards which has $m$ suits and $n$ card types.  For example, a standard deck of $52$ cards has $n=13$ values (Ace, Two, \ldots, King), each appearing $m=4$ times.  We find it helpful to think that $m$ is for \underline{m}ultiplicity and $n$ is for \underline{n}umber of values.  We refer to the elements of $\S_{m,n}$ as permutations, even though for $m>1$ this is technically not the case.  If $X$ is a finite set, we write $\rand{x}\sim X$ to indicate that $\rand{x}$ is chosen uniformly at random from $X$.
	
	Consider the following experiment: a deck with $m$ copies of $n$ different card types is randomly shuffled according to some $\rand{\pi}\sim \S_{m,n}$, and a guesser attempts to guess each card as it is drawn, and the drawn card is discarded after the guess is made (i.e. this is sampling without replacement).  Each time a guess is made, some amount of ``feedback" is given.  For example, one could tell the guesser the true identity of the card they just guessed (the complete feedback model) or they could be told nothing at all (the no feedback model).  This can also be viewed as a one player game where the guesser tries to either maximize or minimize the number of times their guesses are correct, and we will often refer to these models as games.  
	
	These sorts of models were considered by Blackwell and Hodges~\cite{BH} and Efron~\cite{E} in relation to clinical trials. Here, aiming for a fixed number of subjects, say 100, in a medical trial comprising say 4 treatments, a deck of 100 cards with 25 labeled with each treatment is prepared.  Subjects are assigned to treatments as they come into the clinic, sequentially, using the next card (which is then discarded).  Hospital staff has the option of ruling subjects ineligible.  If the staff has strong opinions about the efficacy of treatments and observes which treatments have been given out, they may guess what the next treatment is and bias the experiment by ruling a sickly subject ineligible. It is clearly of interest to be able to evaluate the expected potential bias.  
	
	Card guessing is also a mainstay of classical experiments to test ``Extra Sensory Perception'' (ESP). The most common experiment utilizes a deck of 25 cards where there are five copies of five different types of cards (so $m=n=5$ in our language) where the subjects iteratively try and guess the identity of the next card, and experimenters routinely give various kinds of feedback to enhance ``learning''.   Diaconis~\cite{D} and Diaconis and Graham~\cite{DG} give a review of these problems.
	
	In the no feedback model every strategy guesses $m$ cards correctly in expectation.  The distribution of correct guesses depends on the guessing strategy: if the guesser always guesses the same card type then the variance is 0, and it can be shown that the variance is largest if the guesser uses a permutation of the $mn$ values, see~\cite{DG}.
	
	The complete feedback model is more complicated, but optimal strategies were determined in~\cite{DG}.  Given a strategy $\c{G}$ for the guesser, let $C(\c{G},\pi)$ denote the number of correct guesses the guesser gets in the complete feedback model if they use strategy $\c{G}$ and the deck is shuffled according to $\pi$.  Let $\c{C}_{m,n}^+=\max_{\c{G}}\E[C(\c{G},\rand{\pi})]$, where $\rand{\pi}\sim \S_{m,n}$ and the maximum ranges over all possible strategies $\c{G}$.  Similarly define $\c{C}_{m,n}^-=\min_{\c{G}}\E[C(\c{G},\rand{\pi})]$. The following is proven in~\cite{DG}.
	
	\begin{thm}[\cite{DG}]\label{thm:old complete}
		If $\c{G}^{+}$ (respectively $\c{G}^-$) is the strategy where one guesses a most likely (respectively least likely) card at each step, then $\c{C}_{m,n}^\pm=\E[C(\c{G}^\pm,\rand{\pi})]$.  Moreover, \[\c{C}_{m,n}^\pm=m\pm M_n \sqrt{m}+o_n(\sqrt{m}),\]
		where $M_n=\Theta(\sqrt{\log n})$ is the expected maximum value of $n$ independent standard normal variables.
	\end{thm}
	
	One can also consider variants of these models where $\rand{\pi}$ is chosen according to some non-uniform distribution.  For $m=1$ the case when $\pi$ is obtained from a single riffle shuffle is studied by Ciucu~\cite{C} (no feedback), and Liu~\cite{L} (complete feedback).  Analysis under repeated ``top to random shuffles'' is done by Pehlivan~\cite{Pe}. We emphasize that for our results, we only consider the uniform distribution $\rand{\pi}\sim\S_{m,n}$.
	
	The main focus in this paper is on a feedback model called the partial feedback model, which returns an intermediate amount of information to the guesser.  After each guess, the guesser is only told whether their guess was correct or not (and thus not the identity of the card if they were incorrect).  This feedback protocol was recommended when conducting ESP trials and is a natural notion of bias if card guessing experiments are performed with experimenter and subject in the same room. Given a strategy $\c{G}$ for this game, let $P(\c{G},\pi)$ denote the number of cards the guesser guesses correctly using strategy $\c{G}$ if the deck is shuffled according to $\pi$, and define $\c{P}_{m,n}^+=\max_{\c{G}} \E[P(\c{G},\rand{\pi})]$ and $\c{P}_{m,n}^-=\min_{\c{G}} \E[P(\c{G},\rand{\pi})]$ for the maximum and minimum expected number of correct guesses possible, respectively.
	
	The partial feedback model is significantly more difficult to analyze than the other two models, and relatively little is known about it.  This is in large part due to the fact that we do not understand the optimal strategy in this game, and in particular it is not the case that the strategy $\c{G}^+$ of guessing a maximum likelihood card satisfies $\E[P(\c{G}^+,\rand{\pi})]=\c{P}_{m,n}^+$ for $m\ge 2$, see~\cite{DG}.  We note that bounds for the strategy $\c{G}^+$ in this model were studied recently by Gural, Simper, and So \cite{GSS}.  
	
	Define $\c{N}_{m,n}^{\pm}$ for the no feedback model analogous to how $\c{C}_{m,n}^{\pm}$ and $\c{P}_{m,n}^{\pm}$ were defined, and note that $\c{N}_{m,n}^{\pm}=m$.  One can easily show that $\c{N}_{m,n}^+\le \c{P}_{m,n}^+\le \c{C}_{m,n}^+$ for all $m$ and $n$, with the reverse inequalities holding for $-$ instead of $+$. In particular, by Theorem~\ref{thm:old complete} and the fact that $\c{N}_{m,n}^{\pm}=m$, we obtain $\c{P}_{m,n}^\pm =(1+o(1)) m$ as $m$ goes to infinity, for any fixed $n$.  Motivated by this, our focus for this paper will be in bounding $\c{P}_{m,n}^\pm$ when $m$ is fixed and $n$ is large.  As a point of comparison, we first establish the value of $\c{C}_{m,n}^\pm$ in this regime. Here and throughout we let $\log$ denote the natural logarithm.
	
	\begin{thm}\label{thm:complete}
		For $m$ fixed and $n\rightarrow\infty$, we have
		\[\c{C}_{m,n}^+ = (1+o(1)) H_m \log n,\]
		where $H_m=\sum_{i=1}^{m} j^{-1}$ is the $m$-th harmonic number, and \[\c{C}_{m,n}^-=\Theta(n^{-1/m}).\]
	\end{thm}
	The case $m=1$ of Theorem~\ref{thm:complete} was proved in~\cite{DG}, and we give its simple proof to provide some intuition.  Assuming we always guess a card that is in the deck, the chance of getting the first guess correct is $1/n$, then the second is $1/(n-1)$, and so on.  Thus the expected value is exactly $1+1/2+\cdots+1/n=\log n+O(1)$ as claimed.
	
	Theorem~\ref{thm:complete} shows that for any fixed $m$, in expectation the guesser can achieve arbitrarily many or few correct guesses as $n$ grows in the complete feedback model. In sharp contrast, we show that the guesser cannot obtain arbitrarily many correct guesses in the partial feedback model.

	\begin{thm}\label{thm:main}
		If $n$ is sufficiently large in terms of $m$, we have
		\[\c{P}_{m,n}^+= m+O(m^{3/4}\log^{1/4} m).\]
	\end{thm}
	
	This resolves a 40 year old problem of Diaconis and Graham~\cite{DG}, which was open even for $m=2$ (i.e. a deck with composition $\{1,1,2,2,\ldots,n,n\}$).  In particular, this shows that the information from the partial feedback model is not enough for the guesser to correctly guess asymptotically more cards compared to when they are given no feedback at all. We suspect that the error term in Theorem~\ref{thm:main} can be improved to $m^{1/2 + o(1)}$, which would be best possible; see the discussion in Section~\ref{sec:concluding}.
	
	We conclude this introduction with some brief remarks about the related literature. In the partial feedback model, the enumeration of the number of permutations consistent with a given sequence of guesses can be reduced to the evaluation of certain permanents, see Chung, Diaconis, Graham, and Mallows~\cite{CDGM} and Diaconis, Graham, and Holmes~\cite{DGH}.  These papers contain applications to the partial feedback model, as well as a fascinating ``persistence conjecture'': whenever the guesser guesses a card type $i$ incorrectly, it is optimal for them to continue to guess $i$ in the next step.
	
	Throughout, we focused on evaluating the expected number of correct guesses.  The distribution of the number of correct guesses is treated in~\cite{DG}, see also Proschan~\cite{Pr}. A variety of other feedback mechanisms have also been explored, such as less feedback if the guesser is doing well, and telling the guesser that their guess is ``high'' or ``low'', see Samaniego and Utts~\cite{SU}.
	
	Our evaluation for these models gives one point for each card guessed correctly. It is also natural to consider weighted scores: a correct guess early on might be weighted more heavily than a correct guess towards the end since more information is available to the guesser later on.  This is known as skill scoring and is discussed in~\cite{DG} and Briggs and Ruppert~\cite{BR}.
	
	\textbf{Organization}.  This paper is organized as follows. In Section~\ref{sec:complete}, we prove Theorem~\ref{thm:complete} by analyzing the number of correct guesses made by the maximum (or minimum) likelihood guessing strategy, which is guaranteed to be optimal by Theorem~\ref{thm:old complete}. In Section~\ref{sec:partial}, we prove our main result Theorem~\ref{thm:main} that one cannot do much better than randomly guessing in the partial feedback model. The key ingredient is Lemma~\ref{lem:pointwise}, which shows that in ``typical'' game states, no matter what card the guesser guesses the probability of guessing correctly is at most $(1+o(1))n^{-1}$. Finally, in Section~\ref{sec:concluding} we make some concluding remarks, highlighting some of the many open problems left in this area.
	
	\section{The Complete Feedback Model}\label{sec:complete}
	
	In this section we prove Theorem~\ref{thm:complete}. Throughout this section we treat $m$ as a fixed value, and hence the implicit constants in our asymptotic notation are allowed to depend on $m$.
	
	\begin{lem}\label{lem:compProb}
		For $1\le j\le m$, let $T_j$ be the smallest value $t$ such that $\rand{\pi}_t$ is the $j$th occurrence of some card type.  Equivalently, it is the largest $t$ such that $\{\rand{\pi}_1,\ldots,\rand{\pi}_{t-1}\}$ contains no card type with multiplicity at least $j$.  If $t=\gam n^{1-1/j}$ (here $\gamma > 0$ may depend on $n$), we have
		\begin{align*}
			\Pr[T_j>t]&=1-O(\gam^{j}),\\ 
			\Pr[T_j>t]&=O(\gam^{-j}).
		\end{align*}
	\end{lem}
	We postpone this proof for the moment and show how this implies the result.
	
	\begin{proof}[Proof of Theorem~\ref{thm:complete}]
		We start with the proof of the bounds on $\c{C}_{m,n}^-$ (the lowest expected number of correct guesses possible with complete feedback), and recall from Theorem~\ref{thm:old complete} that this equals $\E[C(\c{G}^-,\rand{\pi})]$ where $\c{G}^-$ is the strategy of guessing a least likely card at each stage.  Let $Y_t$ denote the indicator function for successfully guessing $\rand{\pi}_t$ and let $J_t$ denote the largest multiplicity of a card type appearing in $\{\rand{\pi}_1,\ldots,\rand{\pi}_{t-1}\}$.  Because there are $mn-t+1$ total cards in the deck when one guesses $\rand{\pi}_t$, we have 
		\[
		\Pr[Y_t=1]=\sum_{j=0}^m \Pr[Y_t=1|J_t=j]\Pr[J_t=j]=\sum_{j=0}^m \f{m-j}{mn-t+1}\Pr[J_t=j].
		\]  
		Using $1\le m-j\le m$ for $j<m$, we find
		\[
		\f{\Pr[J_t<m]}{mn-t+1}\le \Pr[Y_t=1]\le \f{m\Pr[J_t<m]}{mn-t+1}.
		\] 
		
		Note that $J_t<m$ if and only if $\{\rand{\pi}_1,\ldots,\rand{\pi}_{t-1}\}$ contains no card with multiplicity $m$, which happens if and only if $T_m\ge t$ (since $T_m$ is the largest $t$ for which this occurs).  Further, $T_m<t$ whenever $t> (m-1)n+1$ (when there are only $n-1$ cards left in the deck one of the cards must have appeared $m$ times), and we always have $t\ge 1\ge 0$.  In total we find that
		\[
		(mn)^{-1}\Pr[T_m\ge t]\le \Pr[Y_t=1]\le m n^{-1}\Pr[T_m\ge t].
		\]
		Because $C(\c{G}^-,\rand{\pi})=\sum Y_t$, we conclude by linearity of expectation that
		\[
		(mn)^{-1}\sum \Pr[T_m\ge t]\le \c{C}_{m,n}^-\le mn^{-1} \sum \Pr[T_m\ge t],
		\]
		so it will be enough to show  $\sum\Pr[T_m\ge t]=\Theta(n^{1-1/m})$. 
		
		For $m=1$ we have $T_m=1$ and the result is immediate, so assume $m\ge 2$.  By Lemma~\ref{lem:compProb}, there exists a sufficiently small constant $c>0$ such that $\Pr[T_m\ge cn^{1-1/m}]\ge \half$.  Using this and the fact that $\Pr[T_m\ge t]\ge \Pr[T_m\ge t+1]$,
		\[
		\sum \Pr[T_m\ge t]\ge c n^{1-1/m} \Pr[T_m \ge c n^{1-1/m}]\ge \half cn^{1-1/m},
		\]
		proving the lower bound. 
		
		For the upper bound, by Lemma~\ref{lem:compProb}, for all $\gam\ge 1$, 
		\[
		\Pr[T_m\ge \gam n^{1-1/m}]\le O(\gam^{-m})\le O(\gam^{-2}).
		\]
		Using $\Pr[T_m\ge t]\le 1$ for $t\le n^{1-1/m}$ and the above inequality, we find
		\begin{align*}
			\sum \Pr[T_m\ge t]&\le n^{1-1/m}+\sum_{p=1}^\infty \sum_{t=p n^{1-1/m}}^{(p+1)n^{1-1/m}-1} O(p^{-2}) \\
			&=n^{1-1/m}+n^{1-1/m}\sum_{p=1}^\infty O(p^{-2})=O(n^{1-1/m}),
		\end{align*}
		giving the desired result. 
		
		We now turn to $\c{C}_{m,n}^+$.  Let $X_t$ be the indicator function of the event that we guess $\rand{\pi}_{t}$ correctly using strategy $\c{G}^+$.  Define $J'_t$ to be the largest multiplicity of a card type in $\{\rand{\pi}_{mn},\ldots,\rand{\pi}_{mn-t+2}\}$ and $T'_j$ the largest value $t$ such that $\{\rand{\pi}_{mn},\ldots,\rand{\pi}_{mn-t+2}\}$ contains no card type with multiplicity at least $j$.  Similar to before we find that $J'_t\ge j$ if and only if $T'_j< t$ and that \[\Pr[X_{mn-t+1}=1]=\sum_{j=1}^m \f{j}{t-1}\Pr[J_{t}'=j]=\rec{t-1}\sum_{j=1}^m \Pr[J_{t}'\ge j]=\rec{t-1}\sum_{j=1}^m \Pr[T_j< t],\]  
		where this last step used that $T_j$ and $T'_j$ have the same distribution.
		Using this, $\sum_{k=1}^N k^{-1}=\log N+O(1)$, and $\log m=O(1)$, we find
		\begin{align*}
			\c{C}_{m,n}^+&=\sum_t \E[X_{mn-t+1}]=\sum_t \sum_j \f{\Pr[T_j<t]}{t-1}=\sum_j \sum_t \sum_{s<t} \f{\Pr[T_j=s]}{t-1}\\&=\sum_j\sum_s \Pr[T_j=s]\sum_{t>s} \rec{t-1}\\&=\sum_j \sum_s \Pr[T_j=s](\log mn-\log s+O(1))\\&=m\log n+O(1)-\sum_j \E[\log T_j].
		\end{align*}
		
		Thus to get the desired result it will be enough to show that $\E[\log T_j] = (1-j^{-1}+o(1))\log n$ for all $j$.  From now on we fix some $1\le j\le m$.  Using summation by parts, we find for any $0<\ep<(2m)^{-1}$ that 
		\begin{align}
			\E[\log T_j]&=\sum_{t=2}^{mn} \log t\Pr[T_j=t]=\sum_{t=2}^{mn} \log t(\Pr[T_j> t-1]-\Pr[T_j>t])\nonumber\\ 
			&=\sum_{t=1}^{mn} \Pr[T_j>t](\log(t+1)-\log t)=\sum_{t=1}^{mn} \Pr[T_j> t]\log(1+t^{-1})\nonumber\\ 
			&=O(1)+\sum_{t=1}^{mn} \Pr[T_j> t] t^{-1}\label{eq:compUp}\\ 
			&=O(1)+\sum_{t\le n^{1-j^{-1}-\ep}} (1-o(1))t^{-1}+\sum_{t>n^{1-j^{-1}-\ep}} \Pr[T_j> t]t^{-1},\label{eq:compLow}
		\end{align}
		where we used $\log(1+t^{-1})=t^{-1}+O(t^{-2})$ to get \eqref{eq:compUp} and Lemma~\ref{lem:compProb} with $\gam\le n^{-\ep}$ to get \eqref{eq:compLow}.  By ignoring the second sum in \eqref{eq:compLow}, we see that $\E[\log T_j]\ge (1-j^{-1}-\ep)\log n+o(\log n)$.  To get an upper bound, we use \eqref{eq:compUp}, the bound $\Pr[T_j>t]\le 1$, and Lemma~\ref{lem:compProb} with $\gam\ge n^\ep$ to get 
		\[
		\E[\log T_j]\le O(1)+\sum_{t\le n^{1-j^{-1}+\ep}} t^{-1}+\sum_{t\ge n^{1-j^{-1}+\ep}} o(t^{-1})=O(1)+ (1-j^{-1}+\ep)\log n+o(\log n).
		\] 
		By taking $\ep$ to be arbitrarily small, we find that $\E[\log T_j] = (1-j^{-1}+o(1)) \log n$ for all $j$, giving the desired result.
	\end{proof}
	It remains to prove Lemma~\ref{lem:compProb}. 
	
	\begin{proof}[Proof of Lemma~\ref{lem:compProb}]
		The first bound is trivial if $t$ is of order $n$, so assume $t=o(n)$. Let $F_j(i)$ with $i\in [n]$ be the event that $\{\rand{\pi}_{1},\ldots,\rand{\pi}_{t}\}$ contains at least $j$ copies of $i$, and let $F_j=\bigcup_i F_j(i)$.  Observe that $T_j\ge t+1$ if and only if $F_j$ does not occur, so it will be enough to show that $\Pr[F_j]=O(\gam^j)$ for $t=\gam n^{1-1/j}$.  Indeed, by a simple counting argument we find
		\[
		\Pr[F_j(i)]=\sum_{j'\ge j} {t\choose j'}{mn-t\choose m-j'}\f{(mn-m)!}{(m!)^{n-1}} \cdot \f{(m!)^n}{(mn)!}=\sum_{j'\ge j} O(t^{j'}n^{-j'})=O(t^j n^{-j})=O(\gam^jn^{-1}),
		\]
		where this second to last step used $tn=o(1)$ when taking the sum.  Taking the union bound over all $i\in [n]$ gives the first result.
		
		For a tuple $x=(x_1,\ldots,x_j)$ with $1\le x_1<\cdots<x_j\le t$, let $A(x)$ be the Bernoulli variable which is 1 if $\rand{\pi}_{x_p}$ is the same value for all $p$, and let $S=\sum_x A(x)$.  Observe that $T_j\ge t+1$ if and only if $S=0$, i.e. if no set of $j$ indices all have the same card type.  Thus it will be enough to show $\Pr[S=0]=O(\gam^{-j})$, which we do by using Chebyshev's inequality.  To this end, let $p_{x}=\Pr[A(x)=1]$ and $p_{x,y}=\Pr[A(x)=A(y)=1]$ for $x\ne y$.  If $\mu$ and $\sig^2$ denote the mean and variance of $S=\sum A(x)$, then by linearity of expectation we get
		
		\begin{align}\label{eq:compVar}
			\sig^2=\E[S^2]-\E[S]^2=\sum_x p_x+2\sum_{x < y} p_{x,y}-\sum_x p_x^2-2\sum_{x < y}p_xp_y\le \mu+2\sum_{x< y}(p_{x,y}-p_{x} p_{y}).
		\end{align} 
		
		To compute $\mu$, note that for all $x$ we have 
		
		\begin{equation}
			p_{x}=\f{m-1}{mn-1}\cdots \f{m-j+1}{mn-j+1}=\f{(m-1)!(mn-j)!}{(m-j)!(mn-1)!}=\Theta(n^{1-j}),\label{eq:px}
		\end{equation}
		and as there are ${t\choose j}$ options for $x$, we have \begin{equation}\label{eq:mu}\mu=\Theta(t^j n^{1-j})=\Theta(\gam^j).\end{equation}
		
		To bound the rest of $\sig^2$, fix some tuple $x$ and let $V_k=\{y:|\{x_1,\ldots,x_j\}\cap \{y_1,\ldots,y_j\}|=k\}$ for $0\le k\le j-1$.  By symmetry, we see that
		\begin{equation}\label{eq:rewrite}
			2\sum_{x< y}(p_{x,y}-p_{x} p_{y})={t\choose j}\sum_{k=0}^{j-1} \sum_{y\in V_k} (p_{x,y}-p_{x} p_{y}).
		\end{equation}
		Thus it will be enough to bound the inner sum for each $k$.  First consider the case $k>0$.  In this case $p_{x,y}$ is the probability that some given $2j-k$ positions of $\rand{\pi}$ take on the same value.  This is 0 if $2j-k>m$, and otherwise by the same reasoning as above
		\[
		p_{x,y}=\f{(m-1)!(mn-2j+k)!}{(m-2j+k)!(mn-1)!}=O(n^{1-2j+k}).
		\]  
		Note that $|V_k|=O(t^{j-k})$, so in total this part of the sum is at most $O(t^{j-k}n^{-2j+k+1})$.  Because $t=O(n)$, this quantity is maximized (in order of magnitude) when $k$ is as large as possible, so we have 
		\[
		\sum_{k=1}^{j-1}\sum_{y\in V_k}(p_{x,y}-p_xp_y)=O(t n^{-j})=O(\gam n^{-j+1-1/j})=O(\gam n^{-j+1-1/m}),
		\]
		where this last step used $j\le m$.
		
		It remains to deal with the case $k=0$. For $y\in V_0$, let $p'_{x,y}$ be the probability that $A(x)=A(y)=1$ and $\rand{\pi}_{x_q}= \rand{\pi}_{y_{q'}}$ for all $q,q'$ and $p''_{x,y}$ the probability that $A(x)=A(y)=1$ and $\rand{\pi}_{x_q}\ne \rand{\pi}_{y_{q'}}$ for any $q,q'$.  Observe that $p_{x,y}=p'_{x,y}+p''_{x,y}$.  By the same reasoning as above, we find that $p'_{x,y}=0$ if $2j>m$ and otherwise it is $\f{(m-1)!(mn-2j)!}{(m-2j)!(mn-1)!}=O(t^j n^{-2j+1})$.  From this and the same reasoning as before, we get
		\[\sum_{y\in V_0} p'_{x,y}=O(k^{2j}n^{-2j+1})=O(t n^{-j+1})=O(\gam n^{-j+1-1/m}).\]
		
		It remains to bound $\sum_{y\in V_0}p_{x,y}''-p_xp_y$, and here we will need to be somewhat careful.  By first conditioning on the event $A(x)=1$, we see that 
		\begin{align*}
			p_{x,y}''&=p_x \cdot \f{mn-m}{mn-j}\cdot \f{m-1}{mn-j-1}\cdots \f{m-j+1}{mn-2j+1}\le p_x \cdot 1 \cdot \f{(m-1)!(mn-2j)!}{(m-j)!(mn-j)!}\\ &\le p_x \f{(m-1)!(mn-2j)!}{(m-j)!(mn)!}(mn)^{j}.
		\end{align*}
		By \eqref{eq:px} we have 
		\[
		p_xp_y\ge p_x \f{(m-1)!(mn-2j)!}{(m-j)!(mn)!}(mn-2j)^j, 
		\]
		and using $|V_0|=O(t^j)$ and \eqref{eq:px} we find
		\begin{align*}
			\sum_{y\in V_0} (p''_{x,y}-p_xp_y)&= O(t^j)\cdot p_x\cdot \f{(m-1)!(mn-2j)!}{(m-j)!(mn)!}\cdot ((mn)^j-(mn-2j)^j)\\&=O(t^{j}\cdot n^{-j+1}\cdot  n^{-2j}\cdot n^{j-1})=O(t^{j}n^{-2j})=O(\gam^jn^{-j-1})=O(n^{-j})=O(n^{-j+1-1/m}),
		\end{align*}
		where this second to last step used that $\gam n^{1-1/j}\le mn$ implies $\gam^j=O(n)$.
		
		In total then by \eqref{eq:compVar},  \eqref{eq:rewrite}, and \eqref{eq:mu}, we have
		\[\sig^2\le \mu+ O(t^j\cdot n^{-j+1-1/m})=\mu +O(\gam^jn^{-1/m})=\mu+o(\mu).\]
		In particular, for $n$ sufficiently large we have $\sig^2\le 4\mu$ (and the asymptotic bound of the lemma is trivial otherwise). Thus by Chebyshev's inequality, we find
		\[
		\Pr[S=0]\le \Pr[|S-\mu|\ge \mu^{1/2} \sig/2]\le 4\mu^{-1}=O(\gam^{-j}),
		\]
		giving the desired result.
	\end{proof}

	\section{The Partial Feedback Model}\label{sec:partial} 
	\subsection{Definitions and Outline}
	Throughout this section we fix a guessing strategy $\c{G}$ and a suitable $\ep=\ep(m)>0$ which will be on the order of $m^{-1/4} \log^{1/4} m$.  Our goal is to prove for large enough $n$ that $\E[P(\c{G},\rand{\pi})] \le (1+ O(\ep))m$. In this section, we simply refer to the partial feedback model as ``the game.''
	
	A {\it history} $h=(g,y)$ of a completed game is a pair of vectors: the $[n]$-valued vector $g$ of all $mn$ guesses made throughout the game, and the boolean vector $y$ of feedback received, so that $y_t = 1$ if and only if the $t$-th card in the deck has value $g_t$. A {\it history at time $t$}, denoted $h_t$, is a truncation of some complete history $h$ to the first $t$ values in each vector, representing all the information available to the guesser after they make the $t$-th guess.
	
	We let $H$ denote a sample of the history of the game given the fixed strategy $\c{G}$ and that the deck is shuffled according to a uniform random $\rand{\pi}\sim \S_{m,n}$. Similarly $H_t$ denotes a sample of the history of the game at time $t$.
	
	Given a history $h=(g,y)$, we write $Y(h)\coloneqq \|y\|$ for the total number of correct guesses, where here and throughout this chapter $\|v\|:=\sum |v_i|$ denotes the $\ell^1$ norm.  Define $a_i(h)\coloneqq |\{t : g_t = i\}|$ to be the number of times card type $i$ has been guessed, and $m_i(h) \coloneqq m - |\{t : g_t = i \text{ and } y_t = 1\}|$ to be the number of copies of card $i$ left to be found in the deck. For a partial history $h_t$, the values $Y(h_t)$, $a_i(h_t)$, and $m_i(h_t)$ are defined in the same way.
	
	We are ready to outline the proof.  The first and most important step is to prove the following ``pointwise'' lemma, which roughly shows that for all typical histories $h_{t-1}$, the probability that the $t$-th guess is correct is at most $(1+o(1))n^{-1}$.
	
	\begin{lem}\label{lem:pointwise}
		For any history $h_{t-1}$ of the game up to time $t-1$ and any $i\in [n]$,
		\[
		\Pr[\rand{\pi}_t = i | H_{t-1} = h_{t-1}]\le \frac{m_i(h_{t-1})}{mn-a_i(h_{t-1}) - Y(h_{t-1})}.
		\]
	\end{lem}
	
	Note that the fraction on the right hand side is a natural estimate for $\Pr[\rand{\pi}_t = i|H_{t-1}=h_{t-1}]$: the numerator is exactly the number of copies of $i$ in the deck that have yet to be found, and the denominator is approximately the total number of positions among $[mn]$ at which such a copy could lie (this may not be exact because $a_i(h_{t-1})$ and $Y(h_{t-1})$ can count the same position twice). We use a simple bijective argument to prove Lemma~\ref{lem:pointwise} in Section~\ref{subsec:pointwise}.
	
	The second step of the proof is to show that the term $Y(H_{t-1})$ in Lemma~\ref{lem:pointwise} is negligible with high probability, which is done by the following lemma.
	\begin{lem}\label{lem:weak}
		For any $0 < \lam \le 1/6$, $n^{1/2}\ge 12\lam^{-1}$, and any fixed strategy $\c{G}$,
		\[\Pr[P(\c{G},\rand{\pi}) > \lam mn]\le 2e^{-mn^{1/2}}.\]
	\end{lem}
	This bound is proved in Section~\ref{subsec:weak} using Lemma~\ref{lem:pointwise} and Chernoff bounds. Combining Lemmas~\ref{lem:pointwise} and~\ref{lem:weak}, and since $Y(h_{t-1})\le Y(h)$, we see that with high probability for any $\ep>0$,
	\[
	\Pr[\rand{\pi}_t = i]\le \frac{m_i(H_{t-1})}{(1-\ep)mn-a_i(H_{t-1})}.
	\]
	We now break guesses into three types, based on how many times a given card $i$ has already been guessed. A guess at time $t$, say with $g_t=i$, is called {\it subcritical} if $a_i(H_{t-1}) < \ep mn$, {\it critical} if $\ep mn \le a_i(H_{t-1})<(1-\ep)mn$, and {\it supercritical} if $a_i(H_{t-1}) \ge (1-\ep) mn$. Note that if even a single supercritical guess is made, then almost all guesses must have been of that same card type, which makes the situation easy to analyze.
	
	By adaptively re-numbering the cards during the game if necessary, we may assume without loss of generality that if there are $k$ card types for which critical guesses are made, then they are exactly the first $k$ cards $1,\ldots, k$.  For any given history $h$, let $b_0(h)$ be the number of subcritical guesses made, let $b_i(h)$, $1\le i \le k$ be the number of critical guesses made with $g_t=i$, and let $b_\infty(h)$ be the number of supercritical guesses made. Define $Y_0(h)$, $Y_i(h)$, and $Y_\infty(h)$ to be the number of correct guesses made in each regime.
	
	We finish the proof by showing with high probability that each of the $Y_i(H)$ values are not much larger than their means.  The subcritical guesses $Y_0(H)$ are handled in Section~\ref{sec:subcritical}, the critical guesses $Y_i(H)$ in Section~\ref{sec:critical}, and the supercritical regime is simple enough to not merit its own subsection.  The proof is then completed in Section~\ref{sec:end-proof}.
	
	Throughout the proof we will often omit floors and ceilings for ease of presentation.  For an event $E$ we let $\ol{E}$ denote its complement.  For real valued random variables $X$ and $Y$, we write $X \succeq Y$ if $X$ stochastically dominates $Y$, i.e. if for all $x\in \mathbb R$, $\Pr[X \ge x] \ge \Pr[Y\ge x]$.  We also recall a standard variant of the Chernoff bound, see for instance \cite{kuszmaul2021multiplicative}.
	
	\begin{lem}\label{lem:chern}
		Let $B(N,p)$ be a binomial random variable with $N$ trials and probability of success $p$.  Then for all $\lam>0$, 
		\[
		\Pr[B(N,p)> (1+\lam) pN]\le e^{-\f{ \lam^2pN}{2+\lam}}.
		\] 
	\end{lem}
	
	\subsection{The Pointwise Lemma}\label{subsec:pointwise}
	In this section we show Lemma~\ref{lem:pointwise}, which is equivalent to an upper bound on the number of $\pi \in \S_{m,n}$ for which at each position up through $t$, either $\pi_t$ is specified or a single value is disallowed for $\pi_t$.  We reduce to the following setup. 	
	\begin{defn}
		Let ${\bf m}=(m_{1},\ldots,m_{n})$ and ${\bf a}=(a_{1},\ldots,a_{n})$
		be vectors of nonnegative integers satisfying $\|{\bf a}\|<\|{\bf m}\|$.
		An \emph{${\bf m}$-permutation} is a word of length $\|{\bf m}\|$ over alphabet $[n]$
		where $i$ appears exactly $m_{i}$ times. An {\it $({\bf m},{\bf a})$-permutation} $\pi$ is an ${\bf m}$-permutation where the first
		$a_{1}$ terms are not $1$, the next $a_{2}$ terms are not $2$,
		and so on, so that exactly $a_{i}$ terms in $\pi$ are forbidden from taking value $i$. 
	\end{defn}
	
	It is significant that $\|a\| < \|m\|$ strictly in the definition of $({\bf m},{\bf a})$-permutations, guaranteeing that no restrictions are made on the value of the last term. Given a history $h_{t-1}$ up to time $t-1$, we let ${\bf m}$ be the vector $(m_1(h_{t-1}),\ldots, m_n(h_{t-1}))$, and ${\bf a}$
	be the vector $(a_1(h_{t-1}),\ldots, a_n(h_{t-1}))$. We claim that the following bound on $({\bf m},{\bf a})$-permutations implies Lemma~\ref{lem:pointwise}.
	\begin{lem}
		\label{lem:main}If $f_{i}({\bf m},{\bf a})$ is the fraction of all
		$({\bf m},{\bf a})$-permutations for which the last term
		is $i$, then
		\[
		f_{i}({\bf m},{\bf a})\le\frac{m_{i}}{\|{\bf m}\|-a_{i}}.
		\]
	\end{lem}
	
	Indeed, by definition $f_i({\bf m},{\bf a})$ is the probability that the last card in $\rand{\pi}$ is exactly $i$ given the current history $h_{t-1}$.
	But all positions past the first $t-1$ are indistinguishable, so
	$f_{i}({\bf m},{\bf a})$ is also the probability that the next card (at index $t$) is $i$.  Thus it suffices to prove Lemma~\ref{lem:main}.
	
	\begin{proof}[Proof of Lemma \ref{lem:main}.]
		It suffices to show the lemma for $i=1$. First we make a technical
		reduction to the case $a_{1}=0$ for convenience.  Let $\tilde{\pi}$ be any sequence of $a_{1}$ cards in which
		$1$ does not appear and $i$ appears at most $m_{i}$ times for all $i>1$. Define
		an \emph{$({\bf m},{\bf a},\text{\ensuremath{\tilde{\pi}}})$-permutation
		}to be an $({\bf m},{\bf a})$-permutation where the first
		$a_{1}$ terms agree with $\tilde{\pi}$. 
		
		Define $f_{i}({\bf m},{\bf a},\tilde{\pi})$ to be the fraction of \emph{$({\bf m},{\bf a},\text{\ensuremath{\tilde{\pi}}})$-}permutations
		which have last term $i$. Since $f_{1}({\bf m},{\bf a})$ is some
		convex combination of the values $f_{1}({\bf m},{\bf a},\tilde{\pi})$, it suffices
		to show that for every specific choice of $\tilde{\pi}$,
		\begin{equation}
			f_{1}({\bf m},{\bf a},\tilde{\pi})\le\frac{m_{1}}{\|{\bf m}\|-a_{1}}.\label{eq:1}
		\end{equation}
		
		Let ${\bf m}'$ be the vector of card counts remaining when the cards
		in $\tilde{\pi}$ are taken out, and let ${\bf a}'=(0,a_{2},a_{3},\ldots,a_{n})$,
		so that an\emph{ $({\bf m},{\bf a},\text{\ensuremath{\tilde{\pi}}})$-}permutation
		is just $\tilde{\pi}$ concatenated with an $({\bf m}',{\bf a}')$-permutation
		$\pi'$. Since $m'_{1}= m_{1}$ and $\|{\bf m}'\|=\|{\bf m}\|-a_{1}$, it suffices to show
		\[
		f_{1}({\bf m}',{\bf a}')\le\frac{m_{1}'}{\|{\bf m}'\|},
		\]
		which is just the case $a_{1}=0$ in the original lemma statement.
		Thus, it remains to show that if $a_{1}=0$, we have
		\begin{equation}
			f_{1}({\bf m},{\bf a})\le\frac{m_{1}}{\|{\bf m}\|}.\label{eq:a1eq0}
		\end{equation}
		
		In fact, we will prove that for any $i$,
		\begin{equation}
			\frac{f_{1}({\bf m},{\bf a})}{f_{i}({\bf m},{\bf a})}\le\frac{m_{1}}{m_{i}}.\label{eq:ratio}
		\end{equation}
		
		The case $i=1$ is trivial, so we just need to prove this for $i>1$, and without
		loss of generality we can assume $i=2$. We divide the $({\bf m},{\bf a})$-permutations $\pi$ which
		end in either $1$ or $2$ into classes as follows. For each $\pi$
		which ends in either $1$ or $2$, consider all positions past the
		first $a_{2}$ which contain either a $1$ or a $2$. Let $S(\pi)$
		denote the set of $\pi'$ obtained by cyclically shifting the $1$'s
		and $2$'s in these positions within $\pi$, fixing all other values.  Note that with this we never move a 1 into a forbidden position (as $a_1=0$) nor a 2 into a forbidden position (as we only shift past the first $a_2$ positions).  It follows that every
		$\pi'\in S(\pi)$ is a $({\bf m},{\bf a})$-permutation ending in 1 or 2.
		
		Note that the total number of $2$'s past the first $a_{2}$ positions
		is exactly $m_{2}$, since every $2$ appears past the first $a_{2}$,
		while the total number of $1$'s past the first $a_{2}$ positions
		is at most $m_{1}$, since there are exactly $m_{1}$ $1$'s in total.
		Thus, we see that the fraction of $\pi'\in S(\pi)$ which end in $1$
		is at most $\frac{m_{1}}{m_{1}+m_{2}}$ for every $\pi$. As the $S(\pi)$
		partition all possible $({\bf m},{\bf a})$-permutations
		$\pi$ which end in either $1$ or $2$, (\ref{eq:ratio}) follows
		for $i=2$.
		
		Finally, to derive (\ref{eq:a1eq0}) it suffices to write (\ref{eq:ratio})
		as
		\[
		\frac{m_{i}}{m_{1}}f_{1}({\bf m},{\bf a})\le f_{i}({\bf m},{\bf a})
		\]
		and sum over $i$, noting that $\sum_{i}f_{i}({\bf m},{\bf a})=1$
		since every $({\bf m},{\bf a})$-permutation must end in some $i$.
	\end{proof}

	\subsection{Weak Bound on $\c{P}_{m,n}^+$ }\label{subsec:weak}

	The next  step is to show that the $Y(h_{t-1})$ term in Lemma~\ref{lem:pointwise} is negligible with high probability. Since $Y(h_{t-1})$ is bounded by just $Y(h)$, the total number of cards guessed correctly, it suffices to show a weak upper bound on the total number of correct guesses in the form of Lemma~\ref{lem:weak}.  To do this we first show the following.

	\begin{lem}\label{lem:binomialStochastic}
		Let  $B_1,\ldots,B_k$ be (not necessarily independent) Bernoulli random variables with $\Pr[B_t=1|\sum_{s<t}B_{s}=x]\le p$ for all $t$ and $x$.  Then $\sum_{t=1}^k B_t$ is stochastically dominated by a binomial random variable $B(k,p)$. 
	\end{lem}
	
	This lemma will be proved by induction. The induction step is the following simple observation.
	
	\begin{lem}\label{lem:dom-induction}
		Let $X,X',Y,Y'$ be integer-valued random variables such that  $X'$ and $Y'$ are $\{0,1\}$-valued, $X\succeq Y$, and for all $x\in \mathbb{Z}$, $(X' | X=x) \succeq (Y' | Y=x)$. Then,
		\[
		X + X' \succeq Y + Y'.
		\]
	\end{lem}
	
	\begin{proof}
		Our goal is to show that for any $y \in \mathbb{Z}$, $\Pr[X + X' \ge y] \succeq \Pr[Y+Y' \ge y]$. But clearly
		\begin{align}
			\Pr[X + X' \ge y] & = \Pr[X \ge y] + \Pr[(X = y-1) \wedge (X' = 1)] \nonumber \\
			& = \Pr[X \ge y] + \Pr[X = y-1] \Pr[X'=1 | X = y-1]\nonumber \\
			& \ge \Pr[X \ge y] + \Pr[X = y-1] \Pr[Y'=1 | Y = y-1] \nonumber\\
			& \ge \Pr[Y \ge y] + \Pr[Y = y-1] \Pr[Y'=1 | Y = y-1]\label{eq:line-critical}\\
			& = \Pr[Y+Y' \ge y].\nonumber
		\end{align}
		Here only (\ref{eq:line-critical}) is worth explaining. Since $X\succeq Y$ we have $\Pr[X \ge y] \ge \Pr[Y \ge y]$ and $\Pr[X \ge y] + \Pr[X = y-1]\ge \Pr[Y \ge y] + \Pr[Y = y-1]$, so by taking convex combinations of these two inequalities, we have for any $t \in [0,1]$, $\Pr[X \ge y] + t \Pr[X = y-1]\ge \Pr[Y \ge y] + t \Pr[Y = y-1]$ as well. Taking $t= \Pr[Y'=1 | Y = y-1]$ completes the proof.
	\end{proof}
	Lemma~\ref{lem:binomialStochastic} follows by iterating Lemma~\ref{lem:dom-induction} with $X=\sum_{s<t} B_{s}$, $X'=B_t$, $Y$ a binomial random variable $B(t-1,p)$, and $Y'$ a Bernoulli random variable with probability $p$.  We omit the details.
	
	We next prove the following, which immediately implies Lemma~\ref{lem:weak}.
	
	\begin{lem}\label{lem:techWeak}
		For any $0 < \lam \le 1/6$, $n\ge 200\lam^{-1}$, and any fixed strategy $\c{G}$,
		\[\Pr[P(\c{G},\rand{\pi}) > \lam mn]\le 2e^{-\lam mn/12}.\]
	\end{lem}
	\begin{proof}
		We first show that few correct guesses are made in the first third of the game, i.e. when $t\le mn/3$.  In this case we apply Lemma~\ref{lem:pointwise} to find that for any $i \in [n]$,
		\[
		\Pr[\rand{\pi}_t = i | H_{t-1} = h_{t-1}]\le \frac{m_i(h_{t-1})}{mn-a_i(h_{t-1}) - Y(h_{t-1})} \le \frac{m}{mn - mn/3 - mn/3} = \frac{3}{n},
		\]
		since up to this point there have been at most $mn/3$ correct guesses and each $i$ has been guessed at most $mn/3$ times. It follows that for $t \le mn/3$, conditional on any $h_{t-1}$, the probability that the $t$-th guess is correct is at most $3/n$. In particular the $t$-th guess is correct with probability at most $3/n$ regardless of the value of $Y(H_{t-1})$, so by Lemma~\ref{lem:binomialStochastic} the number of correct guesses in the first third of the game $Y(H_{mn/3})$ is stochastically dominated by a binomial random variable $B(mn/3, 3/n)$. Applying Lemma~\ref{lem:chern} gives for all $\del\ge 2$,
		\[
		\Pr[Y(H_{mn/3}) > (1+\delta)m] \le \Pr[B(mn/3, 3/n) > (1+\delta)m] \le e^{-\del m/2}. 
		\]
		Taking $\delta = \lam n /4 - 1 \ge \lam n/6\ge 2$ since $n\ge 12\lam^{-1}$, we find
		\begin{equation}\label{eq:dumb1}
			\Pr[Y(H_{mn/3}) > \lam mn/4] \le e^{-\lam mn/12}. 
		\end{equation}

		Let $T$ be the set of $i$ such that $a_i(h_t)< mn/4$ for all $t$, and note that there are at most four card types not in $T$ (since only $mn$ total guesses are made).  Let $E$ be the event that $Y(H_{mn/3})\le \lam mn/4$, and observe that conditional on $E$ we have $Y(H_t)\le (2/3 + \lam/4 )mn$ for all $t$ since at most $2mn/3$ correct guesses can be made in the last $2/3$ of the game. Thus by Lemma~\ref{lem:main} and the above observations, we have for $i\in T$, all $t > mn/3$, and any possible history $h_{t-1}$ for which $E$ occurs,
		\begin{equation}\label{eq:dumb2}
			\Pr[\rand{\pi}_t = i | H_{t-1} = h_{t-1}]\le \frac{m}{mn - mn/4 - (2/3 + \lam/4)mn} \le \frac{24}{n},
		\end{equation}
		where we used $\lam\le 1/6$.
		
		Let $Y'(H)$ denote the total number of correct guesses of card types $i\in T$ and let $Y''(H)$ denote the total number of correct guesses involving $i\notin T$.  Observe that \[Y(H)=Y'(H)+Y''(H)\le Y'(H)+4m\le Y'(H)+\lam mn/2,\] where this last step used $n\ge 8\lam^{-1}$ (which is implicit in our hypothesis of the lemma).
		By \eqref{eq:dumb2} we see that conditional on $E$, $Y'(H) - Y(H_{mn/3})$ is stochastically dominated by a binomial random variable $B(2mn/3, 24/n)$.  Thus
		\begin{align*}
			\Pr[Y(H)>\lam mn]&\le \Pr[Y'(H)>\lam mn/2]\le \Pr[Y'(H) - Y(H_{mn/3}) >\lam mn/4|E]+\Pr[\ol{E}] \\ 
			&\le \Pr[B(2mn/3, 24/n) > \lam mn/4] + \Pr[\ol{E}] \le e^{-\lam mn/12}+e^{-\lam mn/12},
		\end{align*}
		where the last inequality used the Chernoff bound with $\del=\lam n/64-1 \ge \lam n /96\ge 2$  and \eqref{eq:dumb1}.
	\end{proof}
	
	\subsection{Concentration of Subcritical Guesses}\label{sec:subcritical}
	In this section we handle the subcritical guesses.  If $X_t$ denotes the indicator variable that the $t$-th subcritical guess is correct, then intuitively the $X_t$ variables are dominated by Bernoulli random variables with parameter $p=\rec{(1-2\ep)n}$, so the total number of correct subcritical guesses is dominated by a binomial distribution $B(b_0(H), p)$, where we recall that $b_0(H)$ is the number of subcritical guesses in history $H$.  
	
	We would like to say that this  binomial distribution is close to its expectation with high probability. It is not enough, however, to prove this for a fixed binomial distribution. The main technical issue is that the number of trials $b_0(H)$ can be chosen adaptively by the guesser. For example, they can use a strategy where they repeatedly make subcritical guesses until they have guessed an above average number of cards correctly.  This is essentially equivalent to the guesser simulating a summation of Bernoulli random variables $\sum_{t=1}^{mn}B_t$, and then choosing some number of trials $b\le mn$ such that the number of correct subcritical guesses is $\sum_{t=1}^{b}B_t$.  We thus wish to show that for $B_t$ a sequence of independent Bernoulli variables, $\sum_{t=1}^b B_t$ is not much larger than its expectation for all large $b$.  With this, no matter how the guesser chooses $b$, they can never do much better than $pb$.
	
	A weak upper bound for this probability comes from applying the Chernoff bound to all $b\le mn$ and then using a union bound.  Unfortunately when $p$ is very small this upper bound is not effective.  A more careful application of the union bound gives the following technical result, where we think of the $Z_k$'s in its statement as centered binomial random variables with $k$ trials. 
	
	\begin{lem}\label{lem:union}
		Let $0\le p\le 1$, $c,c'>0$, and let $0\equiv Z_0,Z_1,Z_2,\ldots$ be random variables such that $Z_k-Z_{k-1}\ge -p$ for all $k$, and such that for all integers $0\le k'<k$ and all $0<\lam\le 1$,
		\[\Pr[Z_k-Z_{k'}> \lam p(k-k')]\le c' e^{-c\lam^2p(k-k')}.\]
		Then for all $0<\lam\le 1$ and integers $k_1\ge k_0\ge 2\lam^{-1}$, we have
		\[\Pr[\exists k\in [k_0,k_1],Z_k>\lam pk]\le \f{8c'k_1}{\lam k_0} e^{-\rec{256} c\lam^3 pk_0}.\]
	\end{lem}
	\begin{proof}
		Define $\ell=\half \lam k_0\ge 1$.  The idea of the proof is to take a union bound over the events $Z_{\ell a}-Z_{\ell(a-1)}> \lam  p\ell$ for all integers $a\le \f{k_1}{\ell}$, which will turn out to be strong enough to conclude the stated result.  To be precise, let $0=x_0< x_1<\cdots<x_r=k_1$ be any sequence of integers such that $\half \ell\le x_a-x_{a-1}\le \ell$ for all $a>0$, and note that the number of terms in this sequence satisfies \begin{equation}\label{eq:r}r\le \ceil{2k_1/\ell}\le \f{8k_1}{\lam k_0}.\end{equation}  Let $E$ be the event that $Z_{x_b}> \rec{8}\lam pb \ell$ for some $b$.  Observe that $Z_{x_b}=\sum_{a=1}^b Z_{x_a}-Z_{x_{a-1}}$, so $Z_{x_b}> \rec{8} \lam pb \ell$ implies that some $a\le b$ has 
		\[
		Z_{x_a}-Z_{x_{a-1}}> \rec{8} \lam p\ell\ge \rec{8}\lam p(x_a-x_{a-1}).
		\] 
		Thus by the union bound, the hypothesis of the lemma, the fact that $x_a-x_{a-1}\ge \half \ell$, and inequality \eqref{eq:r}, we have 
		\[
		\Pr[E]\le \sum_{a=1}^{r} \Pr[Z_{x_a}-Z_{x_{a-1}}> \rec{8}\lam p(x_a-x_{a-1})]\le r\cdot c'e^{-\rec{128} c\lam^2 p\ell}\le  \f{8c'k_1}{\lam k_0} e^{-\rec{256} c\lam^3 pk_0}.
		\]
		
		We claim that if $Z_k> \lam pk$ for some $k\in [k_0,k_1]$, then $E$ occurs.  Indeed, suppose such a $k$ exists and let $b$ be the smallest integer such that $k\le x_b$, which in particular implies $x_b-k\le\ell$. We also have $b\ge 2$ because $k_0\le k\le b\ell$ and $k_0/\ell= 2\lam^{-1}\ge 2$, and thus 
		\begin{equation}\label{eq:bquart}
			k\ge \half (b-1)\ell \ge \quart b\ell.
		\end{equation}  
		Note that $Z_{x_b} - Z_k \ge - \ell p$ because $Z_k - Z_{k-1} \ge -p$ for all $k$. Using this, $\ell=\half \lam k_0\le \half \lam k$, and inequality \eqref{eq:bquart},  we have 
		\[
		Z_{x_b}> \lam pk-\ell p\ge \half \lam pk\ge \rec{8} pb\ell,
		\]
		so $E$ occurs. Thus,
		\[
		\Pr[\exists k \in [k_0, k_1], Z_k > \lam pk] \le \Pr[E] \le \f{8c'k_1}{\lam k_0} e^{-\rec{256} c\lam^3 pk_0}
		\]
		as desired.
	\end{proof}
	
	Using Lemma~\ref{lem:union}, we can show that subcritical guesses are well behaved.
	
	\begin{lem}\label{lem:subcritical}
		If $\ep\le \rec{8}$ and $n$ is sufficiently large in terms of $\ep,m$, then
		\[
		\Pr\Big[Y_0(H) > (1+4\ep)\frac{b_0(H)}{n}\Big] \le c'\ep^{-2}e^{-c \ep^4 m},
		\]
		for some absolute constants $c,c'>0$.
	\end{lem}
	
	\begin{proof}
		Given $t\le b_0(H)$, let $t'$ be the smallest positive integer for which $b_0(H_{t'}) = t$, so that $t'$ is the time of the $t$-th subcritical guess (note that $t'$ is itself a random variable), and let $X_t \coloneqq Y_0(H_{t'}) - Y_0(H_{t'-1})$. In other words, $X_t$ is the indicator of the $t$-th subcritical guess. Let $E$ be the event that $Y(H) > \ep mn$, and define $E_t$ to be the event that $Y(H_{t'})>\ep mn$. Observe that $\overline{E}$ implies that no $E_t$ occurs. 
		
		Note that $Y_0(H)=\sum_{t=1}^{b_0(H)} X_t$. We modify $Y_0(H)$ to ignore the events $E_t$ as follows. Define $X'_t=X_t$ if $E_{t-1}$ does not occur and $X_t'=0$ otherwise, and let $Y'_0=\sum_{t=1}^{b_0(H)}X'_t$. With $L:=(1+4\ep)\frac{b_0(H)}{n}$, we find
		\begin{align*}
			\Pr[Y_0(H) >L] &\le \Pr[(Y_0 (H) > L) \wedge \overline{E}] + \Pr[E] = \Pr[(Y_0' > L)\wedge \overline{E}] + \Pr[E] \\&\le \Pr[Y_0' > L] + \Pr[E],
		\end{align*}
		
		By Lemma~\ref{lem:weak} we know $\Pr[E]\le 2e^{-mn^{1/2}}$, so for $n$ sufficiently large the contribution of $\Pr[E]$ is negligible.  It remains to upper bound the probability that $Y'_0$ is large.  Note that $X_t' = 1$ if and only if the next term $\rand{\pi}_{t'}$ is exactly the next guess $i$, the total number $a_i(H_{t'-1})$ of times $i$ is guessed is at most $\ep mn$, and the total number $Y(H_{t'-1})$ of correct guesses up to this point is also at most $\ep mn$.
		We now have by Lemma~\ref{lem:pointwise} that
		\[
		\Pr[X_t' = 1|X_1',\ldots,X_{t-1}'] \le \frac{m_i(H_{t'-1})}{mn - a_i(H_{t'-1}) - Y(H_{t'-1})} \le \frac{m}{(1-2 \ep)mn} = \frac{1}{(1-2\ep) n}\eqqcolon p.
		\]
		
		Define $B_1,B_2,\ldots,B_{mn}$ to be independent Bernoulli random variables with $\Pr[B_t=1]=p$ and define $Z_k=\sum_{t=1}^k B_t-pk$.  By the above inequality, we see that given any history $h_{t'-1}$ up to the $t'$-th guess, $X_t'$ is stochastically dominated by $B_t$, and hence $Z_k$ stochastically dominates $\sum_{t=1}^k X'_t-pk$.  Observe that
		\[
		\sum_{t=1}^{b_0(H)}X'_t-pb_0(H)> \f{\ep b_0(H)}{(1-2\ep)n}\iff Y'_0>\f{(1+\ep)b_0(H)}{(1-2\ep)n} \Longleftarrow Y_0'>\f{(1+4\ep)b_0(H)}{n}=L,
		\]
		where the last step used $\ep\le \rec{8}$.  Because $Z_{b_0(H)}$ stochastically dominates the above sum, we have

		\begin{align*}
			\Pr[Y_0'> L]&\le \Pr\l[ \sum_{t=1}^{b_0(H)}X'_t-pb_0(H)> \f{\ep b_0(H)}{(1-2\ep)n}\r]\le \Pr\l[Z_{b_0(H)}>\f{\ep b_0(H)}{(1-2\ep)n}\r]\\ &\le \Pr\l[\exists k\in [\ep mn, mn],\ Z_{k}>\f{\ep k}{(1-2\ep)n}\r],
		\end{align*}
		Where this last step used that the number of subcritical guesses $b_0(H)$ must always be at least $\ep mn$ and at most $mn$.
		
		Because $Z_k$ is a centered binomial distribution, $\Pr[Z_k-Z_{k'}>\lam p(k-k')]\le e^{-\rec{3} \lam^2 p(k-k')}$ for $k'<k$  by Lemma~\ref{lem:chern}, and also $Z_k-Z_{k-1}\ge -p$ for all $k$ by construction.  If $n$ is sufficiently large we have $\ep mn\ge 2\ep^{-1}$, so we can apply Lemma~\ref{lem:union} to the above inequality with $c'=\rec{3}$ and $c=1$ to conclude
		
		\begin{align*}
			\Pr[Y_0'> L]\le \f{8}{(1-2\ep)\ep^2}e^{-\rec{768(1-2\ep)} \ep^4 m}
			\le 16\ep^{-2}e^{-\rec{768} \ep^4 m},
		\end{align*}
		with this last step using $\ep\le \quart$.
	\end{proof}

	\subsection{Concentration of Critical Guesses}\label{sec:critical}
	In the subcritical region we were able to bound the number of correct guesses by a binomial random variable.  For the critical region, we compare the number of correct guesses with a hypergeometric random variable.  We recall that a random variable $S\sim\tr{Hyp}(N,m,b)$ has a hypergeometric distribution (with parameters $N,m,b$) if for all integers $1\le k\le m$ we have \begin{equation}\label{eq:hyp}\Pr[S=k]={b\choose k}{N-b\choose m-k}{N\choose m}^{-1}.\end{equation}
	Equivalently one can define this by uniformly shuffling a deck of $N$ cards with $m$ of these cards being ``good'', and then letting $S$ be the number of good cards one sees in the first $b$ draws from the deck.  From this viewpoint, if we let $R_t$ denote the indicator variable which is 1 if the $t$th draw is a good card, we see that $S=\sum_{t=1}^b R_t$ and that the $R_t$ can be defined by
	\begin{equation}\label{eq:hypR}\Pr[R_t=1]=\frac{m - (R_1 + \cdots + R_{t-1})}{N-t+1}.\end{equation}
	
	We can use the following lemma to bound random variables by hypergeometric random variables.
	
	\begin{lem}\label{lem:domination}
		Suppose $P_1,\ldots, P_k$ and $R_1,\ldots, R_k$ are $\{0,1\}$-random variables satisfying
		\begin{align*}
			\Pr[P_t=1] & \le \frac{m - (P_1 + \cdots + P_{t-1})}{N-t+1} \\
			\Pr[R_t=1] & = \frac{m - (R_1 + \cdots + R_{t-1})}{N-t+1}.
		\end{align*}
		Then $R_1 + \cdots + R_k \succeq P_1 + \cdots + P_k$.
	\end{lem}
	
	The proof of Lemma~\ref{lem:domination} follows from induction and applying Lemma~\ref{lem:dom-induction} with $X=R_1 + \cdots + R_{t-1}$, $X' = R_t$, $Y = P_1 + \cdots + P_{t-1}$, and $Y' = P_t$.  The last thing we need is to use Lemma~\ref{lem:union} in this hypergeometric setting.  
	\begin{lem}\label{lem:hypCon}
		Let $N\ge m^2+m$, define the indicator random variables $R_1,R_2,\ldots,R_N$ as in \eqref{eq:hypR}, and let $S_b:=\sum_{t=1}^b R_t$ for all $b$.  Then for all $b$ and $0<\lam\le 1$, \[\Pr\l[S_b> \f{(1+\lam)bm}{N}\r]\le 3e^{-\f{\lam^2 bm}{3N}}.\]
		Further, for all $0<\lam\le 1$ and integers $b_0,b_1$ satisfying $2\lam^{-1}\le b_0\le b_1\le N$, we have \[\Pr\l[\exists b\in [b_0,b_1],S_b> \f{(1+\lam)bm}{N}\r]\le \f{24b_1}{\lam b_0}e^{-\f{\lam^3 b_0 m}{768N}}.\]
	\end{lem}
	\begin{proof}
		Observe that $S_b\sim \tr{Hyp}(N,m,b)$.  Thus if $q:=b/N$, we have by \eqref{eq:hyp} that \begin{align*}\Pr[S_b=k]&={qN\choose k}{(1-q)N\choose m-k}{N\choose m}^{-1}\le \f{(qN)^k}{k!}\f{((1-q)N)^{m-k}}{(m-k)!}\f{m!}{(N-m)^m}\\ &={m\choose k}q^k(1-q)^{m-k}\l(1+\f{m}{N-m}\r)^m\le {m\choose k}q^k(1-q)^{m-k} e^{m^2/(N-m)}\\&\le {m\choose k}q^k(1-q)^{m-k}\cdot 3,
		\end{align*}
		where this last step used $N-m\ge m^2$.
		
		We thus see that $\Pr[S_b> (1+\lam)qm]\cdot 3^{-1}$ is at most the probability that a binomial distribution with $m$ trials and probability $q$ of success has at least $(1+\lam)qm$ successes, which is at most $e^{-\lam^2 qm/3}$ by Lemma~\ref{lem:chern}.  This gives the first result.
		
		For the second result, define $p=m/N$ and let $Z_b:=S_b-pb$.  Note that $S_b-S_{b'}\sim \tr{Hyp}(N,m,b-b')$ for $b>b'$ (since this is just a sum of $b-b'$ of the $R_t$ variables), so the first result implies \[\Pr[Z_{b}-Z_{b'}> \lam p(b-b')]=\Pr[S_{b}-S_{b'}> (1+\lam)p(b-b')]\le  3e^{-\rec{3} \lam^2 pb}.\]
		We can thus apply Lemma~\ref{lem:union} to the $Z_b$ variables with $c'=3$ and $c=-\rec{3}$ to conclude the result.
	\end{proof}

	Using this we can prove the following.
	
	\begin{lem}\label{lem:critical}
		For $i\ge 1$ finite, $\ep\le \quart$, and $n$ sufficiently large in terms of $\ep,m$, we have 
		\[
		\Pr\Big[Y_i(H) > (1+4\ep)\frac{b_i(H)}{n} + \ep^2 m\Big] \le c' \ep^{-2} e^{-c \ep^4 m}
		\]
		for some absolute constants $c,c'>0$.
	\end{lem} 
	\begin{proof}
		Fix $i$ positive and finite, and let $X_t \coloneqq Y_i(H_{t'}) - Y_i(H_{t'-1})$ where $t'$ is the smallest positive integer for which $b_i(H_{t'}) = t$ (note that $t'$ is itself a random variable). In other words, $X_t$ is the indicator of the $t$-th critical guess of $i$. Define $X_t '=X_t$ if $Y(H) \le \ep mn$ and define $X_t'=0$ otherwise.
		
		Let $R_t$ be random variables as in Lemma~\ref{lem:domination} with $N=(1-2\ep)mn$, and define $S_b=\sum_{i=1}^b R_i$ for all $1\le b\le (1-2\ep)mn$. By applying Lemma~\ref{lem:main} (and noting that $i$ was guessed $\ep mn$ times before its critical guesses started), we see
		\[
		\Pr[X_t'=1] \le \frac{m - (X_1' + \cdots + X_{t-1}')}{(1-2\ep)mn - t + 1}.
		\]
		Thus we can apply Lemma~\ref{lem:domination} with $X'_t$ taking the role of $P_t$, and letting
		\[
		L(b)=(1+4\ep) \frac{b}{n} + \ep^2 m
		\]
		gives
		\begin{align*}
			\Pr[Y_i(H) > L(b_i(H))]&= \Pr\l[ \sum_{t=1}^{b_i(H)} X_t >  L(b_i(H))\r] \\
			& \le \Pr[Y(H) \le \ep mn]\cdot \Pr\l[\sum_{t=1}^{b_i(H)} X_t' > L(b_i(H))\r] + \Pr[Y(H) > \ep mn] \\
			& \le  1\cdot \Pr[S_{b_i(H)} > L(b_i(H))] + 2e^{-mn^{1/2}}\\ 
			& \le  \Pr\l[\exists b\in [1,(1-2\ep)mn],\ S_b >L(b)\r] + 2e^{-mn^{1/2}},
		\end{align*}
		where the second to last step used Lemma~\ref{lem:weak} and the last step used that the value of $b_i(H)$ must lie in $1$ and $(1-2\ep)mn$.  
		
		To bound $\Pr\l[\exists b\in [1,(1-2\ep)mn],\ S_b >L(b)\r]$, we partition $[1,(1-2\ep)mn]$ into intervals $[b_{j-1},b_j]$ (which we define below) and show that  $\Pr\l[\exists b\in [b_{j-1},b_{j}],\ S_b >L_j(b)\r]$ is small for all $j$, where $L_j(b)$ is some quantity upper bounded by $L(j)$.  Taking a union bound will then give the desired result.
		
		Let $b_0:=\half \ep^2(1-2\ep)mn\ge 2$ for $n$ sufficiently large.  By taking $\lam=1$ in Lemma~\ref{lem:hypCon}, we have
		\begin{equation}\label{eq:small}\Pr[\exists b\in [1,b_0],\ S_b > \ep^2 m]=\Pr\l[S_{b_0}>\f{2b_0m}{(1-2\ep)mn}\r]\le 3 e^{-\f{b_0 m}{(1-2\ep)mn}}=3 e^{-\quart \ep^2 m}.\end{equation}
		
		Define $b_j=2^j b_{0}$.  Observe that for all $b\le b_j$ we have $\ep^2\ge 2^{1-j}\f{b}{(1-2\ep)mn}$.  Thus for $j$ such that $2^{1-j}\ge 4 \ep$ we find for $n$ sufficiently large in terms of $j$,
		\begin{align}
			Pr[\exists b\in [b_{j-1},b_j],\ S_b > \f{b}{n}+\ep^2m]&\le \Pr\l[\exists b\in [b_{j-1},b_j],\nonumber\ S_b > (1-2\ep+2^{1-j})\f{bm}{(1-2\ep)mn}\r]\\ &\le \Pr\l[\exists b\in [b_{j-1},b_j],\ S_b > (1+2^{-j})\f{bm}{(1-2\ep)mn}\r]\nonumber\\ &\le 48\cdot 2^je^{-\f{2^j b_0}{2^{3j+9}(1-2\ep)n}}=48\cdot 2^j e^{-2^{-2j-10}\ep^2m}\nonumber\\ &\le 48 \ep^{-1}e^{-2^{-10} \ep^4 m}\label{eq:medium},\end{align}
		where this third inequality used Lemma~\ref{lem:hypCon}.
		
		Let $J=\floor{\log_2(\ep^{-1})}-1$, noting that we can apply the above bound up to \[b_J=2^{J-1}\ep^2(1-2\ep)mn\ge \rec{8} \ep (1-2\ep)mn.\] Observe that for $\ep\le \rec{8}$, \[\f{(1+4\ep)b}{n}\ge  \f{(1+\ep)b}{(1-2\ep)n}.\] Thus by Lemma~\ref{lem:hypCon} applied with $\lam=\ep$, we see that {\small \begin{align*}\Pr\l[\exists b\in [b_J,(1-2\ep)mn],\ S_b > \frac{(1+4\ep)b}{n}\r]&\le \f{24(1-2\ep)mn}{\ep b_J}e^{-\f{\ep^3 b_J}{768(1-2\ep)n}}\le 96\ep^{-2}e^{\f{-1+\ep}{3072(1-2\ep)}\ep^4 m}.\end{align*}}
		Taking the union bound over this, \eqref{eq:small}, and \eqref{eq:medium} for the at most $-\log_2(\ep)\le \ep^{-1}$ values of $j\le J$ gives the result.
	\end{proof}
	\subsection{Completing the Proof}\label{sec:end-proof}
	
	We need the following simple consequence of Lemma~\ref{lem:techWeak}.
	
	\begin{lem}\label{lem:conditional-tail}
		If $n$ is sufficiently large and $A$ is an event with $\Pr[A] = p>0$, then
		\[
		\E[Y(H) | A] < 200m+20 p^{-1}.
		\]
	\end{lem}
	\begin{proof}
		The statement of Lemma~\ref{lem:techWeak} implies the following: For any $0 < \lam \le 1/6$ and $n\ge 200\lam^{-1}$,
		\[\Pr[Y(H)> \lam mn]\le 2e^{-\lam mn/12}.\]
		By taking $x=\lam n$, this is equivalent to saying that for $200\le x\le n/6$ we have
		\[\Pr[Y(H)>xm]\le 2e^{-x m/12}.\]
		In particular, even after conditioning on the event $A$,
		\[
		\Pr[Y(H) > xm | A] \le 2p^{-1}e^{-xm/12}.
		\]
		With this we have
		\begin{align*}
			\E[Y(H) | A] &=m\int_0^n\Pr[Y(H) > xm | A] dx \\&\le 200 m + m\int_{200}^{n/6} \Pr[Y(H) > xm | A] dx+mn\cdot \Pr[Y(H)>mn/6] \\
			& \le 200m + m\int_{0}^{\infty}2p^{-1}e^{-xm/12}dx+2mne^{-mn^{1/2}} \\&= 200m+12p^{-1}+2mne^{-mn^{1/2}},
		\end{align*}
		where the second inequality used the observation made above and Lemma~\ref{lem:weak}.  This gives the result by taking $n$ to be sufficiently large in terms of $m$.
	\end{proof}
	
	Finally we have all the tools to prove the main theorem.
	
	\begin{proof}[Proof of Theorem~\ref{thm:main}]
		We will pick $\ep = O((\log m/m)^{1/4})$, and show that for an appropriate such $\ep$ and $n$ sufficiently large in terms of $m$ and $\ep$, $\E[Y(H)] \le (1+\ep) m$.
		To this end, we define the following three ``atypical'' events: $E_0, E_1$ and $E_\infty$.
		\begin{itemize}
			\item The event $E_0$ is the event that $Y_0(H) > (1+4\ep) b_0(H) / n$, in other words that significantly more than the average number of subcritical guesses are correct.
			\item The event $E_1$ is the event that $Y_i(H) > (1+4\ep) b_i(H) /n + \ep^2 m$ for some $i\ge 1$, in other words that for some critical card $i$, significantly more than the average number of critical guesses of card $i$ are correct.
			\item The event $E_\infty$ is the event that there is at least one supercritical card. In this case, this single card is guessed at least $(1-\ep)mn$ times.
		\end{itemize}
		
		Our goal will be to calculate the conditional expectation of $Y(H)$ depending on whether or not the exceptional events above occur. It will be convenient to group $E_0$ and $E_1$ together and define their union $A = E_0 \vee E_1$. Then,
		{\small \begin{equation}\label{eq:conditional-expectation}
				\E[Y(H)] = \Pr[A]\E[Y(H) | A] + \Pr[\overline{A} \wedge \overline{E}_\infty] \cdot \E[Y(H) | \overline{A} \wedge \overline{E}_\infty] + \Pr[\overline{A} \wedge E_\infty] \E[Y(H) |\overline{A} \wedge E_\infty].
		\end{equation}}
		We first observe that if none of the events $E_0, E_1$ and $E_\infty$ occur, then the conditional expectation of $Y(H)$ is small. Indeed, we have
		\begin{equation}\label{eq:typical}
			\E[Y(H) | \overline{A} \wedge \overline{E}_\infty]=\E[Y(H) | \overline{E}_0 \wedge \overline{E}_1 \wedge \overline{E}_\infty] \le (1+5\ep)m,
		\end{equation}
		since all guesses must be subcritical or critical, and there are at most $\ep^{-1}$ distinct critical card types $i$.
		
		Define $p_j = \Pr[E_j]$ for $j\in \{0, 1\}$. We have by Lemma~\ref{lem:subcritical} and Lemma~\ref{lem:critical} that for some absolute constants $c, c'>0$,
		\[
		p_0 \le c'\ep^{-2}e^{-c \ep^4 m}
		\]
		and
		\[
		p_1 \le c'\ep^{-3} e^{-c \ep^4 m},
		\]
		where there is an extra multiplicative factor of $\ep ^{-1}$ in the second inequality because there may be up to $\ep^{-1}$ critical cards.
		In particular, we have 
		\begin{equation}\label{eq:p-bound}
			\Pr[A]=\Pr[E_0 \vee E_1] \le p\coloneqq 2c'\ep^{-3} e^{-c \ep^4 m}.
		\end{equation}
		By Lemma~\ref{lem:conditional-tail} we find
		\[
		\E[Y(H) | A] \le 200m+20\Pr[A]^{-1},
		\]
		and so
		\[
		\Pr[A]\E[Y(H) | A] \le 200m \Pr[A] + 20 \le 200m p +20,
		\]
		for $p$ defined above.  By picking an appropriate $\ep = O((\log m/m)^{1/4})$, we find
		\begin{equation}\label{eq:firstTerm}
			\Pr[A]\E[Y(H) | A] \le m^{-\Omega(1)} + 20 < \ep m
		\end{equation}
		for $m$ sufficiently large.
		
		Finally, to control the third term of (\ref{eq:conditional-expectation}), note that if there is a supercritical card, at most $m$ guesses are correct for that card (since there are a total of $m$ copies of that card in the deck), and at most $\ep mn$ guesses are made of any other card, so all other guesses are subcritical. In particular, including guesses of the unique supercritical card, there at most $b_0(H)\le 2\ep mn$ subcritical guesses. Thus, by the definition of $E_0$, we get
		\[
		\E[Y(H) |  \overline{A}\wedge E_\infty ] \le m + (1+4\ep) (2\ep m) \le (1+3\ep)m.
		\]
		
		In total, using (\ref{eq:conditional-expectation}), \eqref{eq:typical}, \eqref{eq:firstTerm}, and the inequality above, we find that for $m, n$ sufficiently large,
		\[
		\E[Y(H)] \le \ep m + \Pr[\overline{A}] \cdot (1 + 5\ep) m \le (1+6 \ep) m= m+O(m^{3/4} \log^{1/4} m),
		\]
		completing the proof.
	\end{proof}

	\section{Concluding Remarks}\label{sec:concluding}
	In this paper we proved results for two different feedback models.  The first were bounds on $\c{C}_{m,n}^\pm$, which is the most/least number of cards one can guess in the complete feedback model.  For fixed $m$ we determined $\c{C}_{m,n}^+$ asymptotically, and for $\c{C}_{m,n}^-$ we gave the correct order of magnitude.
	\begin{quest}
		What is $\c{C}_{m,n}^-$ asymptotically?
	\end{quest}
	We note that a more careful analysis of the proof of Theorem~\ref{thm:complete} gives that this value is asymptotic to $n^{-1}\E[T_m]$ with $T_m$ as in Lemma~\ref{lem:compProb}, so it suffices to compute this expectation.
	
	Our main result was proving a tight asymptotic upper bound on $\c{P}_{m,n}^+$, which is the most number of cards one can guess correctly in expectation in the partial feedback model.  Specifically, we proved $\c{P}_{m,n}^+=m+O(m^{3/4}\log m)$ provided $n$ is sufficiently large.  One consequence of this bound is the following.  For $\pi \in \S_{m,n}$, let $L(\pi)$ denote the largest integer $p$ such that there exist $i_1<\cdots< i_p$ with $\pi_{i_j}=j$ for all $1\le j\le p$.  Define $\c{L}_{m,n}=\E[L(\rand{\pi})]$ where $\rand{\pi}\sim \S_{m,n}$.
	
	\begin{cor}\label{cor:L}
		If $n$ is sufficiently large in terms of $m$, we have
		\[\c{L}_{m,n}\le  m+O(m^{3/4}\log m).\]
	\end{cor}
	\begin{proof}
		Consider the strategy $\c{G}$ in the partial feedback model which guesses 1 until you guess one correctly, then 2 until you guess one correctly, and so on; and if you guess $n$ correctly, play arbitrarily for the remaining trials.  Then $P(\c{G},\rand{\pi})\ge L(\rand{\pi})$, and the result follows from Theorem~\ref{thm:main}.
	\end{proof}
	
	In the original version of this paper, we conjectured that $\c{L}_{m,n}=m-o(m)$ provided $n$ is sufficiently large in terms of $m$.  This has since been proven by Clifton et.\ al.~\cite{CDHSY} in a strong form: they determine the exact formula for $\c{L}_m:=\lim_{n\to \infty} \c{L}_{m,n}$ for all $m$, and they show \[\left|\c{L}_m-\left(m+1-\frac{1}{m+2}\right)\right|\le O(e^{-\be m}),\]
	where $\be>$ is some absolute constant.

	A trivial lower bound for $\c{P}_{m,n}^+$ is $m$ obtained by guessing a card type uniformly at random at each trial.  A more complicated strategy gives $m+\Om(m^{1/2})$ corrects guesses in expectation.  We rigorously prove this in \cite{DGS}, and here we give a brief sketch of the strategy.  Guess 1 a total of $mn/2$ times.  If you guessed at least $\half m+\sqrt{m}$ cards correctly, guess 2 for the rest of the game, otherwise keep guessing 1.  In the latter scenario we always get exactly $m$ correct guesses.  One can show that the first scenario happens with some constant probability, and given this the expected number of 2's left in the second half of the deck is at least $m/2$, and in total this gives a lower bound of $m+\Om(m^{1/2})$.  We suspect that this lower bound is close to the truth.
	\begin{conj}
		For all $\ep>0$ and $n$ sufficiently large,
		\[\c{P}_{m,n}^+=m+O(m^{1/2+\ep}).\]
	\end{conj}
	The current proof overshoots this bound at two points.  The first is in Lemma~\ref{lem:union} where we try and bound the probability that an ``adversarial'' binomial distribution deviates significantly from its mean.  Our proof of this lemma essentially only used a union bound, and it's plausible that more sophisticated techniques could decrease this error term.
	
	The second point is in the bounds of Lemmas~\ref{lem:subcritical} and \ref{lem:critical} where we bound the probability that the subcritical or critical guesses are much larger than average.  We note that by adding in an error term of $\ep m$ to the lower bound of $Y_0$ in Lemma~\ref{lem:subcritical}, one can decrease the probability from roughly $e^{-\ep^4 m}$ to $e^{-\ep^3 m}$, so the central issue is the critical case, and it seems like new ideas are needed here.
	
	Another problem of interest is bounding $\c{P}_{m,n}^-$, the fewest number of cards one can guess correctly in expectation in the partial feedback model.  In \cite{DGS} we prove $\c{P}_{m,n}^-\le m-\Om(m^{1/2})$ using an analog of the strategy for $\c{P}_{m,n}^+$.  We also prove an asymptotic lower bound for $\c{P}_{m,n}^-$ of $1-e^{-m}$ by showing that one always has probability at least this of guessing at least one card correctly.  Thus $\c{P}_{m,n}^-=\Om(1)$, which is again in sharp contrast to the complete feedback model where one can get arbitrarily few correct guesses in expectation.  There is still a large gap between these bounds, and as in Theorem~\ref{thm:main} we suspect that the partial feedback model does not allow one to guess significantly fewer guesses than in the no feedback model.
	\begin{conj}
		If $n$ is sufficiently large in terms of $m$, then
		\[\c{P}_{m,n}^-\sim m.\]
	\end{conj}
	The central difficulty in this setting is that there does not exist an analog of Lemma~\ref{lem:pointwise} which lower bounds $\Pr[\rand{\pi}_t=i]$ given that we have not guessed $i$ many times and that we have guessed few cards correctly.  For example, say we incorrectly guessed 1 a total of $(m-1)n$, so the remaining cards are $m$ copies of 1.  Then the probability that the next card is 2 is 0 despite the fact that we have not guessed 2 at all nor guessed any cards correctly. 
	
	\section*{Acknowledgments}
	
	The authors would like to thank Steve Butler, Sourav Chatterjee, Andrea Ottolini, Mackenzie Simper, and Yuval Wigderson for fruitful conversations.

\end{document}